\documentclass[a4paper,12pt]{article}
\usepackage[utf8]{inputenc}
\usepackage[english]{babel}
\usepackage[ruled,section]{algorithm}
\usepackage{makeidx}
\usepackage[font=small]{caption}
\usepackage{amsmath}
\usepackage{amsthm}
\usepackage{amsfonts}
\usepackage{amssymb}
\usepackage{mathrsfs}
\usepackage{graphicx}
\usepackage[margin=25mm]{geometry}
\usepackage{enumerate}
\usepackage{indentfirst}
\usepackage[none]{hyphenat}
\usepackage{color}
\usepackage{authblk}
\makeindex

\usepackage{graphicx}
\usepackage{epstopdf}

\usepackage{ifthen}
\usepackage{color}

\newtheorem{theorem}{Theorem}[section]
\newtheorem{remark}[theorem]{Remark}
\newtheorem{lemma}[theorem]{Lemma}
\newtheorem{proposition}[theorem]{Proposition}

\newcommand{\supp}{\operatorname{supp}}

\begin{document}

\title
{The Nehari manifold for a degenerate logistic parabolic equation }
\author{ Juliana Fernandes
\,and\, Liliane A. Maia
}

\date{} 

\maketitle

\begin{abstract}

\noindent The present paper analyses the behavior of solutions to a degenerate logistic equation with a nonlinear term of the form $b(x)f(u)$, where the weight function $b$ is assumed to be nonpositive. We exploit variational techniques and comparison principle in order to study the evolutionary dynamics. A crucial role is then played by the Nehari manifold, as we note how it changes as the parameter $\lambda$ in the equation or the function $b$ vary, affecting the existence and non-existence of stationary solutions. We describe a detailed picture of the positive dynamics and also address the local behavior of solutions near a nodal equilibrium, which sheds some further light on the study of the evolution of sign-changing solutions. 
\medskip



\end{abstract}


\section{Introduction}\label{intro} 

Our goal in this paper is to study the 
solutions of the following semilinear parabolic equation
\begin{align}\label{eq0}
	\begin{cases}
		\partial_t u=\Delta u+\displaystyle{\lambda u + b(x) |u|^{\nu -1}u},\quad &(x,t) \in \Omega \times (0,+ \infty),\\
		u \rvert_{\partial\Omega} =0,\quad &t \in (0,+\infty)\\
		u\rvert_{t=0} = u_0(x),\quad &x \in \Omega,
	\end{cases}
\end{align}
where $\Omega$ is an open smooth bounded domain 
in $\mathbb{R}^{N}$, $N\geq 2$, $\lambda$ is a real positive parameter, 
$1 < \nu < 2^*-1$, where $2^* =+ \infty$ if $N=2$, or $2^*= 2N/(N-2)$ if $N \geq 3$,
and $b$ is a continuous function satisfying $b(x) \leq 0$ and $b(x)=0$ in an open proper subset $\Omega_0$ of $\Omega$, with positive Lebesgue measure and smooth boundary.

The main purpose is to apply variational methods and comparison principle in order to analyze the behavior of solutions as the initial data varies in the phase space. The Nehari manifold $\mathcal{N}$ contained in the abstract space, associated to the energy functional of the stationary elliptic problem, will be used to locate the stationary solutions, and identify convergence regions of  evolutionary trajectories.
We note how the different ranges of $\lambda$ and the sign of the weight function
$b(x)$ deeply affect the global dynamics picture. 
To our knowledge, this is the first time the so-called Nehari approach is applied to address the asymptotic analysis of solutions to the logistic problem.

From the Population Dynamics point of view, this class of equations appears as an interplay of two well known classical laws: Malthusian and Verhulst (logistic) growth. In that context, $u(x,t)$ models the evolution of the distribution of a single species in the inhabiting area $\Omega$, $\lambda$ is related to the growth rate of the population $u$, $b(x)$ translates the crowding effects within the region $\Omega\setminus \Omega_0$. The region $\Omega_0$ is also referred to as the favorable region, as it represents the region where $u$ is allowed to enjoy exponential growth, as expected by the Malthus law. The analysis of the evolution of this problem may be found, for instance in the very complete and interesting work by J\'ulian L\'opez-Gomes \cite{LG05}.

The analysis of the asymptotic behavior of positive solutions for this class of degenerate logistic equations was also done in \cite{Arrieta}; see also references therein for some other related contributions.

Although the interest in non-negative solutions justifies itself from the biological motivation given above, sign-changing solutions have also been investigated for a wide class of reaction-diffusion equations. 
In that direction, some work has already been done for the analysis of solutions. Under very general dissipative conditions on the nonlinearity, it was obtained in \cite{RBVL} the existence of two extremal equilibria, which gives bounds for the asymptotic dynamics and therefore for the related global attractor. The result can also be applied to a large class of degenerate logistic equations. Also stability of such solutions is addressed in \cite{Kaji}.

On the analysis of the behavior of solutions to parabolic PDE's, some alternative tools have been used. Besides the usual comparison principle and sub/super solution, variational methods have also been applied in the theory. More precisely, the Nehari manifold $\mathcal{N}$ was proved to be important also for the study of the evolution dynamics, as it can be used as a borderline separating regions of global existence and blow-up.

More recently, the authors have proposed in \cite{FeMa} an analysis of the interplay of variational methods and dynamical systems, using the Nehari manifold to give a rather complete picture of the abstract phase space, for reaction-diffusion equations with asymptotically linear growth. The study was inspired by the ideas developed in \cite{GW05} and \cite{DMSW11}, for semilinear heat equations with the presence of finite time blow-up.

The related semilinear stationary elliptic equation has been extensively addressed using different approaches. General questions on existence, multiplicity, and nonexistence of positive solutions have been discussed. The main tools that have been applied are bifurcation, sub and super-solutions or minimization and linking methods
\cite{Ou1, AT, DPF}.
The existence problem for sign-changing solutions is a bit more delicate with just a few results in this setting to be found in the literature \cite{Ou1, DPF}. We tackle the existence problem of such solutions using Ambrosetti and Rabinowitz \cite{AR} Mountain Pass Theorem  on $\mathcal{N}$.

The paper is organized as follows. In Section 2 we present the stationary problem, and the geometric features of the associated Nehari manifold.
In Section 3, we prove existence, nonexistence of a positive solution under this new approach and obtain a sign-changing solution.
Finally, Section 4 is devoted to the parabolic dynamics. We get uniform boundedness of trajectories, and local behavior near the nodal solution.

\section{The stationary problem}

Firstly, we discuss existence and multiplicity of positive and nodal solutions for the stationary problem, as it will help us to get a better ideia of the global dynamic structure of solutions on the phase space. 

We consider the Hilbert space $H_{0}^{1}(\Omega)$ with its
standard scalar product and norm%
\[
\left\langle u,v\right\rangle :=\int_{\Omega}\nabla u\cdot\nabla
v,\text{\qquad}\left\Vert u\right\Vert :=\left(  \int_{\Omega}\left\vert \nabla u\right\vert ^{2}\right)  ^{1/2}.
\]
and the associated stationary elliptic problem
\begin{equation}\label{eq2elliptic}
	\begin{cases}
		-\Delta u=\lambda u +  b(x) |u|^{\nu - 1}u, \quad x \in \Omega,\\
		u \rvert_{\partial\Omega} =0, \quad u \in H_{0}^{1}(\Omega).
	\end{cases}
\end{equation}

Recall that the eigenvalues for the negative  Laplacian in a bounded domain,
with Dirichlet boundary condition, are given by
$	0<
	{\lambda_1(\Omega)}< {\lambda_2(\Omega)}<\cdots< {\lambda_n(\Omega)}<\cdots, \,\,\,\, {\lambda_j(\Omega)} \rightarrow\infty \mbox{ as } j\rightarrow\infty,
$
with eigenfunctions $\phi_j$.
Similarly, if  $\Omega_0 \not = \emptyset$, we denote by $\lambda_1(\Omega_0)$ the first eigenvalue of $- \Delta$ in the subset $\Omega_0$ also with Dirichlet boundary condition.

Using the method of sub and super solution, together with bifurcation arguments, 
Ouyang in \cite{Ou1} obtained a complete description on existence and non-existence of positive solutions of the elliptic nonlinear problem.
Notice that, since $b(x)\leq 0$ and  not identically zero, the condition $\int_\Omega b(x) \phi_1^{\nu+1}dx <0$ assumed  in  \cite{AT} is automatically satisfied. 

We summarize in the next theorem the classical results found in \cite{Ou1} and \cite{AT} about existence and uniqueness of positive solutions of the elliptic problem.
Using variational methods, we prove  in Lemma \ref{ATLemma2.2} the nonexistence statement, based  in \cite{AT}.

\begin{theorem}\label{OuyTheorem2}
	Assume that $b \leq 0$$\;(\not \equiv 0)$ is a  continuous function on $\Omega $ and 
	$$f(\lambda,x,u):= \lambda u + b(x) |u|^{\nu -1}u$$ 
	Then  it holds that, 
	\begin{itemize}
		\item[(i)] If $\Omega_0 = \emptyset$, then for every $\lambda >0$ there exists a unique positive solution $u$ of  problem \eqref{eq2elliptic}.
		\item[(ii)] If $\Omega_0 \not = \emptyset$, then for every $\lambda <\lambda_1(\Omega_0))$ there exists a unique positive solution $u$ of  problem \eqref{eq2elliptic}, and
		for every  $\lambda \geq \lambda_1(\Omega_0)$,  problem \eqref{eq2elliptic} admits no positive solution.
		%
	\end{itemize}
\end{theorem}
\begin{lemma}\label{ATLemma2.2}
	Problem \eqref{eq2elliptic} does not admit a positive solution for any $\lambda \geq \lambda_1(\Omega_0)$.
\end{lemma}
\begin{proof}
	Suppose by contradiction there exists $u$ a positive solution of problem  \eqref{eq2elliptic}  and multiply the equation by $\phi_1^0$, the positive first eigenfunction of  $-\Delta$ in $\Omega_0$.
	Integrating by parts on this open set with smooth boundary $\partial \Omega_0$, we obtain
	\[
	\int _{\partial \Omega_0}u\frac{\partial \phi_1^0}{\partial \eta}
	-(\lambda -\lambda_1(\Omega_0))\int_{\Omega_0}u \phi_1^0
	- \int_{\Omega_0} b(x)|u|^{\nu-1}u \phi_1^0 =0,
	\]
	where $\eta $ is the normal exterior unitary vector on $\partial \Omega_0$.
	On the other hand, $\int_{\Omega_0} b(x)|u|^{\nu-1}u \phi_1^0=0$ and
	\[
	\int _{\partial \Omega_0}u\frac{\partial \phi_1^0}{\partial \eta}<0 \quad \text{and }
	\quad  (\lambda -\lambda_1(\Omega_0))\int_{\Omega_0}u \phi_1^0> 0,
	\]
	which yields a contradiction. 
\end{proof}

In order to apply some variational arguments, we consider the
functional associated with the equation in \eqref{eq2elliptic}, $I:H_{0}^{1}(\Omega)\rightarrow\mathbb{R}$, given by
\begin{equation}\label{energia}
	I(u)= \frac{1}{2}\|u\|^{2} - 
	\frac{1}{2}	\int_{\Omega} \lambda  u^2
	- 	\frac{1}{\nu+1}\int_{\Omega}b(x)|u|^{\nu + 1}.
\end{equation}
which is of class $\mathcal{C}^{2},$ with derivative
\begin{equation}\label{derivadaenergia}
	I^{\prime}(u)v=\int_{\Omega}  \nabla u\cdot\nabla v
	-\int_{\Omega} \lambda u v -  \int_{\Omega} b(x)|u|^{\nu - 1}u v,\qquad u,v\in H_{0}^{1}(\Omega).
\end{equation}

The functional $J:H_{0}^{1}(\Omega) \rightarrow\mathbb{R}$
given by%
\begin{equation}\label{neharifunction}
	J(u):=I^{\prime}(u)u=\int_{\Omega} |\nabla
	u|^{2}  -\int_{\Omega} \lambda u^2 -  \int_{\Omega} b(x)|u|^{\nu +1}
\end{equation}
is of class $\mathcal{C}^{1}$ and defines 
the so-called Nehari manifold%
\[
\mathcal{N}:=\{u\in H_{0}^{1}(\Omega):\text{ \ }%
J(u)=0\}.
\]
We also consider the complementary sets
\begin{equation}\label{n+}
	\mathcal{N}_+:=\{u\in H_{0}^{1}(\Omega):u\neq0,\text{ \ }%
	J(u)>0\}
\;\;
\text{and}
\;\;
	\mathcal{N}_{-}:=\{u\in H_{0}^{1}(\Omega):u\neq0,\text{ \ }%
	J(u)<0\}.
\end{equation}
So if $u \in \mathcal{N}$, substituting $J(u)=0$
in the functional $I$, gives
\begin{align}\label{Ionnehari}
	I(u)
	&=\left( \frac{1}{2}- \frac{1}{\nu +1}\right)
	\int_\Omega (|\nabla u|^2 - \lambda u^2)\;dx
	=\left( \frac{1}{2}- \frac{1}{\nu +1}\right)
	\int_\Omega b(x)|u|^{\nu + 1}\;dx.
\end{align}


The points in the Nehari manifold $\mathcal{N}$ correspond to critical points of the maps 
\begin{equation}\label{fiber}
\phi_u: t\mapsto I(tu):= \frac{t^2}{2}\int_{\Omega} |\nabla
u|^{2}  -\int_{\Omega}  F(\lambda,x,tu)\;,
\end{equation}
where $F(\lambda, x ,u)= \int_0^uf(\lambda, x, s) ds$, and so it is natural to divide
$\mathcal{N}$ into three subsets corresponding to local minima, local maxima and points of inflexion of fibrering maps $\phi_u$. Notice that 
\[
\phi_u^{''}(t)=\int_\Omega (|\nabla u|^2 - f_u(\lambda,x,tu)u^2)\;dx.
\]
It is  then naturally defined the following sets, 
\[
\mathcal{S}^+=\left \{ u \in \mathcal{N} : \int_\Omega (|\nabla u|^2 -  f_u(\lambda,x,u)u^2)\;dx >0 \right \},
\]
\[
\mathcal{S}^-=\left\{ u \in \mathcal{N} : \int_\Omega (|\nabla u|^2 -  f_u(\lambda, x,u)u^2)\;dx <0 \right \},
\]
and
\[
\mathcal{S}^0=\left\{ u \in \mathcal{N} : \int_\Omega (|\nabla u|^2 -  f_u(\lambda, x,u)u^2)\;dx =0 \right \}.
\]
Moreover, it follows that they can be rewritten,
\begin{align} \label{Splus}
	\nonumber
	\mathcal{S}^+
	&=\{ u \in \mathcal{N} : (1- \nu) \int_\Omega b(x)|u|^{\nu + 1} \;dx >0\}\\
	&=\{ u \in \mathcal{N} : \int_\Omega b(x)|u|^{\nu + 1} \;dx <0\}.
\end{align}
Similarly
$\mathcal{S}^-=\{  u \in \mathcal{N} : \int_\Omega b(x)|u|^{\nu + 1} \;dx >0\}$	 and
$\mathcal{S}^0=\{ u \in \mathcal{N} :\int_\Omega b(x)|u|^{\nu + 1} \;dx =0\}$.

\begin{remark}\label{signI}
	If  $u \in \mathcal{N}$, from \eqref{Ionnehari} and \eqref{Splus} it holds that  $u \in \mathcal{S}^+$ if and only if $I(u) <0$.
	One similarly gets that , $u \in \mathcal{S}^-$ and $u \in \mathcal{S}^0$ if and only if $I(u) >0$ and $I(u) =0$, respectively.
\end{remark}

In what follows, based in \cite{BZ}, we need to consider the subsets
\[
L^+:= \left\{u \in H_0^1(\Omega): \Vert u\Vert=1, \int_\Omega(|\nabla u|^2 - \lambda u^2)\;dx >0 \right \}
\]
and similarly $L^-$ and $L^0$, replacing $>$ by $<$ and $=$, respectively.
We also define
\[
B^+:= \left \{ u \in H_0^1(\Omega): \Vert u\Vert=1, \int_\Omega b(x) |u|^{\nu +1}\;dx >0\right \}
\]
and $B^-$ and $B^0$ analogously.
The next  proposition explores the role played by $b(x) \leq 0$ on this setting.

\begin{proposition}\label{Tsupport}
	 $(\it i)$ If  $\lambda_1(\Omega)<\lambda < \lambda_1(\Omega_0)$,  then $\overline {L^-} \cap {B^0}  = \emptyset$,
	and \\
	$(\it ii)$
	if $\lambda_1(\Omega_0)< \lambda$,  then $\overline {L^-} \cap {B^0} \not = \emptyset$.  
\end{proposition}
\begin{proof}
	Suppose $w \in \overline {L^-} \cap {B^0} $, so by definition $w \not= 0$ and 
	\begin{equation}\label{support}
		0=
		\int_{\Omega} - b(x) |w|^{\nu+1}\geq \int_{\Omega_0} - b(x) |w|^{\nu+1}.
	\end{equation}
	We claim that  the support of $w$ is contained in the closure of $\Omega_0$. Indeed, w.l.o.g. assume
	$w$ is continuous, and there is $x_0 \in \Omega \setminus \overline\Omega_0$ such that 
	$|w(x_0)|=\delta >0$. 
	Then, there is a small ball $B_\varepsilon(x_0) \subset \Omega \setminus \overline\Omega_0$ inside which $|w(x)| \geq \delta/2$ and thus, by \eqref{support}
	\[
	0=
	\int_{\Omega}  -b(x) |w|^{\nu+1} \geq 
	\int_{B_\varepsilon(x_0)}  -b(x) |w|^{\nu+1} \geq
	\int_{B_\varepsilon(x_0)} - b(x) |\frac{\delta}{2}|^{\nu+1}>0,
	\]
	which is impossible. So $supp\{w\} \subset \overline\Omega_0$.
	But then, if $\lambda < \lambda_1(\Omega_0)$,
	\begin{equation*}
		0<\int_{\Omega_0}  (\lambda_1(\Omega_0)-\lambda)  w^2
		\leq
		\int_{\Omega_0}   | \nabla w |^2 -\lambda  w^2
		= 	\int_{\Omega}   | \nabla w |^2 -\lambda  w^2
		\leq 0,
	\end{equation*}
	since $w \in \overline {L^-}$, which gives a contradiction, and hence $\overline {L^-} \cap {B^0}  = \emptyset$.
	
	In case $\lambda_1(\Omega_0)< \lambda$,	let $\phi_1^0$ be the positive (normalized) eigenfunction associated with the first eigenvalue $\lambda_1(\Omega_0)$. Then, the support of $\phi_1^0$ is equal to $\overline {\Omega_0}$ and
	$
	\int_{\Omega}  b(x) |\phi_1^0|^{\nu+1}= \int_{\Omega_0}  b(x) |\phi_1^0|^{\nu+1}=0,
	$
	and so $\phi_1^0 \in B^0$. Moreover,
	$\phi_1^0 \in L^-$  since
	\[
	\int_{\Omega}   | \nabla \phi_1^0 |^2 -\lambda  (\phi_1^0)^2=
	(\lambda_1(\Omega_0)-\lambda)\Vert \phi_1^0\Vert_{L^2(\Omega_0)} <0.
	\]
\end{proof}

We present the next result found in \cite[Theorem 4.2]{BZ}o9, with minor adaptations. We include a proof for the sake of completeness.
\begin{theorem}\label{BZTheorem4.2}
	Suppose  $\overline {L^-} \cap {B^0}  = \emptyset$.  Then
		$(\it i)$  $\mathcal{S}^0 = \{0\}$;
		$(\it ii)$
		 $\mathcal{S}^- = \emptyset$;
		$(\it iii)$
		$\mathcal{S}^+$ is bounded;
\end{theorem}
\begin{proof}
	$(i)$ Suppose $u_0 \in \mathcal{S}^0 \setminus \{0\}$.
	Then $\frac{u_0}{\Vert u_0\Vert} \in L^0 \cap B^0 \subset \overline {L^-} \cap {B^0} = \emptyset$, which is impossible. Hence $\mathcal{S}^0=\{0\}$.\\
	$(ii)$ Since $b(x) \leq 0$, then by \eqref{Splus} $\mathcal{S}^- = \emptyset$.\\
	$(iii)$ Suppose $\mathcal{S}^+$ is unbounded. Then, there exists a sequence $\{u_n\} \in \mathcal{S}^+$ such that
	\begin{equation} \label{Splusbounded}
		\int_\Omega(|\nabla u_n|^2 - \lambda u_n^2)\;dx=  \int_\Omega b(x) |u_n|^{\nu +1}\;dx<0
	\end{equation}
	and $\Vert u_n\Vert \to \infty$.
	Take $v_n:=\frac{u_n}{\Vert u_n \Vert}$, then $v_n \rightharpoonup v$ in $H_0^1(\Omega)$ and, by Sobolev compact embedding, $v_n\rightarrow v $ in $L^q(\Omega)$, for $2 \leq q < 2^*$, and $v_n(x)\rightarrow v(x)$ a.e. in $\Omega$.
	Dividing \eqref{Splusbounded} by $\Vert u_n\Vert^2$ yields
	\begin{equation}\label{ineqvn}
	\int_{\Omega}   | \nabla v_n |^2 -\lambda  v_n^2\;dx = \int_{\Omega} b(x) |v_n|^{\nu+1}\Vert u_n\Vert^{\nu -1}\;dx.
	\end{equation}
	Since the left hand side is uniformly bounded but $\Vert u_n\Vert \to \infty$, it follows that\\
	$\lim_{n \to \infty} \int_{\Omega} b(x) |v_n|^{\nu+1}\;dx =0$, and hence 
	$ \int_{\Omega} b(x) |v|^{\nu+1}\;dx =0$.
	Now we claim that $v_n \to v$ strongly in $H^1_0(\Omega)$.
	Indeed, if $v_n$ does not converge strongly to $v$, 
	\begin{equation} \label{wlsc}
		\int_{\Omega}   | \nabla v |^2 -\lambda  v^2
		<
		\liminf_{n \to \infty} \int_{\Omega}   | \nabla v_n |^2 -\lambda  v_n^2\leq0.
	\end{equation}
	Hence $v\not = 0$ and $\frac{v}{\Vert v \Vert} \in  {L^-}$. 
	We conclude that  $\frac{v}{\Vert v \Vert} \in L^- \cap B^0 \subset  \overline {L^-} \cap {B^0}$, 
	contradicting the assumption.
	So $v_n \to v$, $\Vert v \Vert =1$, and $v \in B^0$.
	Moreover, by \eqref{ineqvn}
	\[
	\int_{\Omega}   | \nabla v |^2 -\lambda  v^2 = 	\liminf_{n \to \infty} \int_{\Omega}   | \nabla v_n |^2 -\lambda  v_n^2 \leq 0,
	\]
	which implies $v \in \overline {L^-}$. Thus $v \in \overline {L^-} \cap B^0$, which is an absurd. This completes the proof.
\end{proof}

The delicate study of the different subsets in the complement of $\mathcal{N}$ are fundamental to our developments later in the parabolic setting.
We follow the literature and denote, for $k \in \mathbb{R}$,
$
I^k(u)=\{u\in H_0^1(\Omega): I(u)<k  \}.
$

\begin{lemma}\label{GWLemma5}
	Suppose $\overline {L^-} \cap {B^0} = \emptyset$. Then
	for any $k>0$, it holds that $I^k \cap \mathcal{N}_+$ is bounded in $H_0^1(\Omega)$.
\end{lemma}

\begin{proof}
	Case 1)
	Let $u\in \mathcal{N}_+$ and suppose there exists $t_u>0$ such that $t_uu \in \mathcal{S}^+$, which means $\int_{\Omega}b |t_uu|^{\nu +1} <0$.
	Then $\int_{\Omega}b |u|^{\nu +1} <0$ , i. e. $\frac{u}{\Vert u \Vert} \in B^-$.
	Moreover, since	$J(u)>0$, then $t_u<1$, because $I(tu)$ is decreasing in the variable $t$ up to $t=t_u$.\\
	Since $ \mathcal{S}^+$ is bounded, by Theorem \ref{BZTheorem4.2}(iii), then there exists $M_\lambda> 0$ such that
	$
	\Vert t_u u\Vert < M_\lambda.
	$
	If we show that there is $T > 0$, uniform in $u$ in this case, such that $T < t_u < 1$, then $\Vert u \Vert < M_\lambda / T$ and we conclude the proof. If not, there is a sequence $(u_n) \subset \mathcal{N}_+$, such that $t_n u_n \in \mathcal{S}^+$ and $t_n \to 0$.
	Since $\Vert t_n u_n \Vert < M_\lambda$, then $\Vert u_n \Vert$ may go to infinity. Assume, by contradiction, that this is the case so that
	there exists a sequence $(u_n) \subset I^k \cap \mathcal{N}_+$, for which there exist $t_n<1$ satisfying $t_n u_n \in \mathcal{S}^+$, and such that $\Vert  u_n \Vert\rightarrow +\infty $, and take $v_n:=\frac{u_n}{\Vert u_n \Vert}.$ Then $v_n \rightharpoonup v$ in $H_0^1(\Omega)$ and, by Sobolev compact embedding, $v_n\rightarrow v $ in $L^q(\Omega)$, for $2 \leq q < 2^*$, and $v_n(x)\rightarrow v(x)$ a.e. in $\Omega$.\\
	Moreover, since $J(t_{u_n}u_n)=0$,
	\begin{align}\label{k2}
		\nonumber
		k> I(u_n)&=I(u_n)-\frac{1}{(\nu+1)\;t_{u_n}^{\nu +1}}J(t_{u_n}u_n)\\
		\nonumber
		&=\frac{1}{2} \int_{\Omega}   | \nabla u_n |^2 -\lambda  u_n^2-\frac{1}{\nu+1}\int_{\Omega} b(x) |u_n|^{\nu+1} \\
		\nonumber
		&-\left(
		\frac{1}{(\nu+1)\;t_{u_n}^{\nu +1}} t_{u_n}^2\int_{\Omega}   | \nabla u_n |^2 -\lambda u_n^2-\frac{1}{(\nu+1)\;t_{u_n}^{\nu +1}}t_{u_n}^{\nu+1}\int_{\Omega} b(x) |u_n|^{\nu+1} \right )\\
		& =  (\frac{1}{2} - \frac{1}{(\nu+1)\;t_{u_n}^{\nu -1}}) \int_{\Omega}   | \nabla u_n |^2 -\lambda  u_n^2 .
	\end{align}
	By (\ref{k2}) and $-1/ t_n^{\nu -1} \to -\infty$, for $n$ sufficiently large we have 
	\[
	\int_{\Omega}   | \nabla u_n |^2 -\lambda  u_n^2
	=
	t_n^{\nu -1} 	\int_{\Omega} bu_n^{\nu+1} <0,
	\]
	which yields  $v_n \in L^- \cap B^-$.
	Then, dividing (\ref{k2}) by $\Vert u_n\Vert^2$, we obtain
	\[
	\frac{k}{\Vert u_n\Vert^2}>   (\frac{1}{2} - \frac{1}{(\nu+1)\;t_n^{\nu -1}}) \int_{\Omega}   | \nabla v_n |^2 -\lambda  v_n^2 >0.
	\]
	Now, if $v_n$ does not converge to $v$, 
	\[
	\int_{\Omega}   | \nabla v |^2 -\lambda  v^2
	<
		\liminf_{n \to \infty} \int_{\Omega}   | \nabla v_n |^2 -\lambda  v_n^2=0.
	\]
	Similarly,
	\begin{align}\label{k1}
		\nonumber
		k> I(u_n)&=I(u_n)-\frac{1}{2\;t_{u_n}^2}J(t_{u_n}u_n)\\
		\nonumber
		&=\frac{1}{2} \int_{\Omega}   | \nabla u_n |^2 -\lambda  u_n^2-\frac{1}{\nu+1}\int_{\Omega} b(x) |u_n|^{\nu+1} \\
		\nonumber
		&-\left(
		\frac{1}{2\;t_{u_n}^2} t_{u_n}^2\int_{\Omega}   | \nabla u_n |^2 -\lambda  u_n^2-\frac{1}{2\;t_{u_n}^2}t_u^{\nu+1}\int_{\Omega} b(x) |u_n|^{\nu+1} \right )\\
		& =  (\frac{t_{u_n}^{\nu-1}}{2} - \frac{1}{\nu+1}) \int_{\Omega} b(x) |u_n|^{\nu+1}.
	\end{align}
	Moreover,
	since $t_n \to 0$ and $\int_{\Omega} b(x) |u_n|^{\nu+1}<0$, for $n$ sufficiently large, by (\ref{k1})
	\[
	\frac{k}{\Vert u_n \Vert^{\nu +1}}
	>
	(\frac{t_n^{\nu-1}}{2} - \frac{1}{\nu+1}) \int_{\Omega} b(x) |v_n|^{\nu+1}>0,
	\]
	so
	\[
	\lim_{n \to \infty} \int_{\Omega} b(x) |v_n|^{\nu+1}=
	\int_{\Omega} b(x) |v|^{\nu+1}=0.
	\]
	Hence $v\not = 0$ and $\frac{v}{\Vert v \Vert} \in \overline {L^-} \cap {B^0}$, which contradicts the hypothesis. 
	Consequently, $v_n \to v$ and $\Vert v \Vert =1$. Moreover
	\[
	\int_{\Omega}   | \nabla v |^2 -\lambda  v^2 =0= \int_{\Omega} b(x) |v|^{\nu+1}, 
	\] 
	so $\frac{v}{\Vert v \Vert} \in L^0 \cap B^0 \subset  \overline {L^-} \cap {B^0}$, again a contradiction since $\overline {L^-} \cap {B^0} = \emptyset$. Hence $0<T < t_n$ and the boundedness of $u$ in $\mathcal{N}_+$ is true in this case.
	\\
	Case 2) Let $u\in \mathcal{N}_+ \cap I^k$ and $u$ non-projectable in $\mathcal{S}^+$, which implies $u \in L^+ \cap B^-$.  Then 
	we have
	\begin{equation}\label{k4}
		k> I(u) >I(u)-\frac{1}{(\nu+1)}J(u)
		>  (\frac{1}{2} - \frac{1}{(\nu+1)}) \int_{\Omega}   | \nabla u |^2 -\lambda  u^2 >0.
	\end{equation}
	Supose, by contradiction, that in this case
	there exists a sequence $(u_n) \subset I^k \cap \mathcal{N}_+$, such that $\Vert  u_n \Vert\rightarrow +\infty $, and take $v_n:=\frac{u_n}{\Vert u_n \Vert}.$ Then $v_n \rightharpoonup v$ in $H_0^1(\Omega)$ and, by Sobolev compact embedding, $v_n\rightarrow v $ in $L^q(\Omega)$, for $2 \leq q < 2^*$, and $v_n(x)\rightarrow v(x)$ a.e. in $\Omega$. 
	Then, dividing (\ref{k4}) by $\Vert u_n\Vert^2$, we obtain
	\[
	\frac{k}{\Vert u_n\Vert^2}
	>  (\frac{1}{2} - \frac{1}{(\nu+1)}) \int_{\Omega}   | \nabla v_n |^2 -\lambda  v_n^2 >0
	\]
	and taking the limit as $n \to \infty$,
	if $v_n$ does not converge to $v$, 
	\begin{equation} \label{wlsc}
		\int_{\Omega}   | \nabla v |^2 -\lambda  v^2
		<
		\liminf_{n \to \infty} \int_{\Omega}   | \nabla v_n |^2 -\lambda  v_n^2=0.
	\end{equation}
	Hence $v\not = 0$ and $\frac{v}{\Vert v \Vert} \in  {L^-}$, and since $u/\Vert u \Vert \in B^-$,
	$
	\int_{\Omega} b(x) |v|^{\nu+1} \leq 0.
	$
	On the other hand,  it holds
	\begin{align}\label{b2}
		\nonumber
		\frac{k}{\Vert u_n\Vert^2}>\frac{1}{\Vert u_n\Vert^2} I(u_n)
		&=
		\frac{1}{2} \int_{\Omega}   | \nabla v_n |^2 -\lambda  v_n^2-\frac{1}{(\nu+1)\Vert u_n\Vert^2}\int_{\Omega} b(x) |u_n|^{\nu+1} \\
		&\geq 
		\frac{1}{2} \int_{\Omega}   | \nabla v_n |^2 -\lambda  v_n^2-\frac{1}{(\nu+1)}\int_{\Omega} b(x) |v_n|^{\nu+1}. 
	\end{align}
	Taking the limit in (\ref{b2}), as $n \to \infty$, and using (\ref{wlsc}) and the fact that $ v_n \to v$ in $L^{\nu + 1}(\Omega)$, it follows that
	$
	0 \leq \frac{1}{(\nu+1)}\int_{\Omega} b(x) |v|^{\nu+1}.
	$
	Therefore $\int_{\Omega} b(x) |v|^{\nu+1} = 0$, and then 
	$\frac{v}{\Vert v \Vert} \in L^- \cap B^0 \subset  \overline {L^-} \cap {B^0}$, giving again a contradiction since $\overline {L^-} \cap {B^0} = \emptyset$. 
	So, $v_n \to v$ strongly, $\Vert v \Vert =1$, and it holds
	\[
	\int_{\Omega}   | \nabla v |^2 -\lambda  v^2 = \int_{\Omega} b(x) |v|^{\nu+1}=0, 
	\] 
	which means $\frac{v}{\Vert v \Vert} \in L^0 \cap B^0 \subset  \overline {L^-} \cap {B^0}$, again a contradiction since $\overline {L^-} \cap {B^0} = \emptyset$.
	This completes the proof of the lemma.
\end{proof}

\begin{lemma}\label{newLemma5}
Suppose $\overline {L^-} \cap {B^0} = \emptyset$. If $u\in \mathcal{N}_-$ then $u$ is projectable on $\mathcal{S}^+$, i. e. there exists $t_u$ such that $t_u u \in \mathcal{S}^+$. Moreover,
$I(u)<0$ and the set $\mathcal{N}_-$ is bounded in $H^1_0(\Omega)$.
\end{lemma}
\begin{proof}
Since $u\in \mathcal{N}_-$, 
$
J(u)=\int_{\Omega}  | \nabla u |^2 - \lambda |u|^2 -
\int_{\Omega} b(x) |u|^{\nu+1} <0.
$
If $\supp \{u\} \subset \Omega_0$, 
then $\frac{u}{\Vert u \Vert}\in B_0$, and also 
$\int_{\Omega}  | \nabla u |^2- \lambda |u|^2 <0$, which implies $\frac{u}{\Vert u \Vert} \in L^-$. This is a contradiction with the assumption $\overline {L^-} \cap {B^0} = \emptyset$. 
Hence $\supp \{u\} \cap (\Omega \setminus \Omega _0) \neq \emptyset$. 
In this case, for $t>0$, we  take $J(tu)$ in \eqref{neharifunction},
since the coefficients of $t^2$ is negative and of $t^{\nu+1}$ is positive,  there exists a unique $t_u >1$, such that $t_u u \in \mathcal{S}^+$.
By Theorem \ref{BZTheorem4.2} $(iii)$, $\mathcal{S}^+$ is bounded, so there is $C>0$ such that
$\Vert t_u u \Vert < C$. Therefore $\Vert u \Vert < C/t_u < C$.
Moreover, using $I(u) < I(u) -1/2 J(u)$, then
\begin{equation*}
	I(u)=\frac{1}{2} \int_{\Omega}   | \nabla u |^2 -\lambda  u^2-\frac{1}{\nu+1}\int_{\Omega} b(x) |u|^{\nu+1} \\
	 <  \int_{\Omega}  (\frac{1}{2} - \frac{1}{\nu+1})b(x) |u|^{\nu+1}<0.
\end{equation*}
\end{proof}
\begin{remark}\label{Inegative}
It follows that, if $\overline {L^-} \cap {B^0} = \emptyset$, then the set of functions
$(I^0 \cap \mathcal{N}_+)\cup \mathcal{S}^+ \cup \mathcal{N}_-$ is bounded in $H_0^1(\Omega)$.
\end{remark}

The next result provides a characterization for projectable functions on $\mathcal{N}\setminus \{0\}$, which coincides with $\mathcal{S}^+$ if $\lambda_1(\Omega)< \lambda <\lambda_1(\Omega_0)$. In other words, since $b(x) \leq 0$, we describe the only possible geometry in order to have a turning point for the fibrering map.
\begin{proposition}\label{projectu0}
	Let $\lambda_1(\Omega)< \lambda <\lambda_1(\Omega_0)$ and $u_0 \in H_0^1(\Omega) \setminus \{0\}$.
	There exists $\bar \alpha>0$ such that $\bar \alpha u_0 \in \mathcal{S}^+$ if and only if  
	$\frac{u_0}{\Vert u_0\Vert} \in  L ^- \cap B^-$.
\end{proposition}
\begin{proof}
		First let us prove that 	$\frac{u_0}{\Vert u_0\Vert} \in  L ^- \cap B^-$ is a sufficient  condition. Since 
	$	\int_{\Omega}   | \nabla u_0 |^2 -\lambda  u_0^2 $ 
	and
	$\int_\Omega b(x)|u_0|^{\nu + 1}$
	are both negative, the fibrering map $\phi_{u_0}$ has exactly one turning point at
	$$t(u_0)= \displaystyle  \left [
	\frac{\int_{\Omega}   | \nabla u_0 |^2 -\lambda  u_0^2}{\int_{\Omega}  b(x) |u_0|^{\nu+1}} 
	\right]^{\frac{1}{\nu -1}},
	$$
	that is $t(u_0)u_0 \in \mathcal{S}^+$.
	
		Conversely, suppose there exists $\bar \alpha>0$ such that $\bar \alpha u_0 \in \mathcal{S}^+$. 
	Then $\frac{\bar \alpha u_0}{\Vert \bar \alpha u_0\Vert} \in L^-$, and hence $\frac{u_0}{\Vert u_0\Vert} \in  L ^-$ . Since $\overline {L^-} \cap {B^0}  = \emptyset$, by Proposition \ref{Tsupport},  and $B^+ = \emptyset$, then 
	$\frac{u_0}{\Vert u_0\Vert} \in  L ^- \cap B^-$. So, the proof of this lemma is complete.
	
%
%
\end{proof}

\section{Existence and nonexistence results}

The next  theorem improves  the classical results in \cite{Ou1, AT, BZ} on the existence of a positive solution, providing a sharp upper bound for the range of $\lambda$. 

We define
$
	\underline d:=\inf_{u\in\mathcal{S}^+}I(u) , 
	$
	with $ -\infty \leq \underline d.$
	By Remark \ref{signI},  it holds $\underline d <0$ if $\mathcal{S}^+$ is not empty.

\begin{theorem}\label{Ibounded}
If  $\lambda_1(\Omega)< \lambda <\lambda_1(\Omega_0)$ ,  then $I$ is bounded  from below in $H^1_0(\Omega)$,
and there exists a minimizer $\varphi >0$ such that $I(\varphi)= \underline d$.
\end{theorem}
\begin{proof}	
Take any $u \not = 0$ in $H^1_0(\Omega)$. If there exists $t> 0$ such that $t\, u \in \mathcal{N}$, then $t \,u \in \mathcal{S}^+$ by Theorem \ref{BZTheorem4.2}, and hence $I(u) \geq I(tu) \geq \underline d$.
On the other hand, if there exists no such $t$, we claim that $I(tu) \geq 0$ for all $t> 0$. Indeed, 
by Proposition \ref{projectu0}, $\frac{u}{\Vert u \Vert} \in \overline{L^+}\cup B^0$.
Therefore, there are two cases.
If $\frac{u}{\Vert u \Vert} \in \overline{L^+}$, then $I(tu) \geq 0$ for all $t>0$.
If $\frac{u}{\Vert u \Vert} \in {B^0}$, then $\frac{u}{\Vert u \Vert} \in L^+$, 
since  $\overline {L^-} \cap {B^0}  = \emptyset$. In this case $I(tu) \geq 0$ for all $t>0$.
We conclude that $I(u) \geq \underline d$, for any $u \in H^1_0(\Omega)$.\\
Moreover, by  \cite[Theorem 4.4 ]{BZ}, since $\phi_1 \in L^-$, $B^+=\emptyset$
and  $\overline {L^-} \cap {B^0} = \emptyset$, it follows that  there exists a minimizer $\varphi$ of $I(u)$
on $\mathcal{S}^+$ which is also a minimizer in the whole space, because $\mathcal{N}$ is a natural constraint.\\
W. l.o.g. $\varphi \geq 0$, since $I(\varphi)=I(|\varphi|)$, then $|\varphi|$ is an interior minimum and so also a solution. 
Supose $\varphi(x_0)=0$ for some $x_0 \in \Omega$.  By the Hopf Lemma this is impossible, hence $\varphi> 0$, and by the uniqueness of the positive solution it is the Ouyang solution.
\end{proof}
\begin{remark}\label{BrezisKato}
Observe that since $f(\lambda,x,u)= \lambda u + b(x) |u|^{\nu -1}u \leq C(1+\vert u \vert^{\nu})$, for $\nu +1\leq 2^*$, then we can apply the essential Brezis-Kato Lemma 
 and obtain that a weak solution $u$ of \eqref{eq2elliptic} is in $C^{1,\alpha}_{loc}(\Omega)$, for any $\alpha <1$. If $\partial \Omega \in C^2$, then $u \in C^{1,\alpha}(\overline \Omega)$, and additionally, if $b \in C^{0, \alpha}(\Omega)$, then
$u \in C^{2,\alpha}(\Omega) \cap C(\overline \Omega)$ is a classical solution of problem \eqref{eq2elliptic}.
\end{remark}
The next result relies on Remark \ref{BrezisKato} and can be found in \cite[Lemma 5.2]{Kaji} , and suits our settings.

\begin{lemma} \label{kaji}
Let $\lambda_1(\Omega)< \lambda <\lambda_1(\Omega_0)$.
The unique positive stationary solution $\varphi$ is isolated from other stationary solutions with respect to the  $H^1_0(\Omega)$ topology. Similarly for the negative solution $- \varphi$.
\end{lemma}
Regarding the trivial solution, it was proved in \cite[Theorems 4.2 and 4.5]{RBVL} that it is an isolated equilibrium point and it  is known to be unstable in the subset of nonnegative initial data for $\lambda_1(\Omega) < \lambda$, see for instance \cite{Arrieta}.

In order to obtain another solution, we employ the Mountain Pass Theorem of Ambrosetti and Rabinowitz \cite{AR}.  Recall that
a sequence $(u_n)$ in $H_0^1(\Omega)$ is said to be Palais Smale at $c$ for $I$, and denoted by $(PS)_c$, if $I(u_n)\rightarrow c$ and $I'(u_n)\rightarrow 0$.


\begin{theorem}\label{MPT}
Let $\tilde{I}(u):=I(u)-\underline d$. For $\varphi>0$ and $-\varphi<0$ local minima on $\mathcal{S}^+$, it holds that:
	$(\it i)$ $\tilde{I}(\varphi)=0$;
	$(\it ii)$ there exists $\rho$ and $\delta>0$ such that $\tilde{I}(u)\geq \delta>0$, for any $u\in B_\rho(\varphi)\cap \mathcal{N}$;
	$(\it iii)$ $\tilde{I}(-\varphi)=0$ and $\rho<\| \varphi - (-\varphi) \|=2\| \varphi \|$;
that is, $\tilde{I}$ satisfies the geometrical hypotheses of  the Mountain Pass Theorem on $\mathcal{N}$. \\
Moreover $\tilde{I}$ satisfies   $(PS)_c$ condition at 
\[
c = \inf_{\gamma \in \Gamma} \max_{0 \leq t \leq 1} I(\gamma(t)),
\]
where
$\Gamma= \{ \gamma:[0,1] \to \mathcal{N}\;: \gamma(0)= \varphi, \gamma(1)=-\varphi\}$,
and so
there exists a  nontrivial solution $u^*$ of \eqref{eq2elliptic}satisfying $I(u^*)= c >\underline d$.
\end{theorem}

In order to prove Theorem \ref{MPT}, we need the following  lemma.

\begin{lemma}\label{EVP}
\label{lem:PS}Every $(PS)_{c}$-sequence $(u_{k})$ for $I$ on
$\mathcal{N}$, with $c \not = 0$, contains a subsequence which is a $(PS)_{c}$-sequence for
$I$.
\end{lemma}
%
\begin{proof}
 We evoke the proof of Lemma 2.5 in \cite{ClMa}.
For the functional $J$ we have
$|J'(u_k)u_k|\leq \| \nabla J(u_k)\| \| u_k\|$.
We claim that $|J'(u_k)u_k|\rightarrow \rho \geq 0$ and additionally that $\rho >0$. Indeed, since $J(v)=0$
for $v\in \mathcal{N}$, then
\begin{equation*}
	J'(v)\cdot v=  2\int (|\nabla v|^2-\lambda v^2)dx- (\nu+1)\int b(x)|v|^{\nu+1}
	=  [2-(\nu+1)]\int b(x)|v|^{\nu+1}\geq 0.
\end{equation*}
Also, for $\| u_k\|\leq M$ with $u_k\rightharpoonup u$ and $u_k\in \mathcal{S}^+$, we have $u_k\rightarrow u$ in $L^{\nu+1}$ and $u_k(x)\rightarrow u(x)$. 

Therefore, we get that
\begin{equation*}
	\lim_{k\rightarrow\infty} |J'(u_k)u_k|=|2-(\nu+1)|\lim_{k\rightarrow\infty} \left|\int b(x)|u_k|^{\nu+1}dx\right|= (\nu-1)\left|\int b(x)|u|^{\nu+1}\right|=\rho
\end{equation*}
We then divide into two cases. First, suppose $u\equiv0$. 
From \eqref{Ionnehari} it holds that 
$I(u_k)\rightarrow 0$, which is also not possible, since $I(u_k)\rightarrow c \not = 0$.

For the second case $u\not\equiv 0$, the convergences $u_k\rightharpoonup u$ and $u_k\rightarrow u$ in $L^{\nu+1}$, and if $\rho=0$, would lead  to 
$\frac{u}{\|u\|} \in \overline{L^-} \cap B^0 =\emptyset$.
The contradiction on both cases $u\equiv 0$ and $u\not \equiv 0$ gives us $\rho>0$.

\end{proof}

\begin{proof}[Proof of Theorem \ref{MPT}]
The Nehari manifold may be written as $\mathcal{N}=\mathcal{S}^0\cup \mathcal{S}^+=\{ J^{-1}(0) \}$, which is  closed in $H_0^1(\Omega)$. That might allow us to apply Ekeland Variational Principle on $\mathcal{N}$, which is a closed metric space. In fact, since $I:\mathcal{N}\rightarrow \mathbb{R}\cup \infty$ is continuous and bounded from below by Theorem \ref{Ibounded}, 
$
0>I(u)\geq \underline d> - \infty .
$

As a consequence, item $(ii)$ can be proved as follows. We suppose by contradiction that for all fixed $\rho$ with $0<\rho<2\|\varphi\|$, there exists a sequence $(u_n)\subset \mathcal{N}\cap \partial B_\rho(\varphi)$ such that $I(u_n)\rightarrow \underline d$. That is,  $(u_n)$ is a minimizing sequence, and then from Ekeland Variational Principle we would get the existence of $(v_n)\in \mathcal{N}$ with $I(v_n)\rightarrow \underline d$, $\| v_n-u_n \| \rightarrow 0$ and $I|_{\mathcal{N}}^{'} (v_n)\rightarrow 0$. Hence, by Lemma \ref{EVP} and
since $I$ satisfies $(PS)_{c}$  (see  \cite[Lemma 2.1]{DPF} ), $v_n\rightarrow v$, up to a subsequence. Then $I(v)=\underline d$, $\tilde{I}'(v)=0$ and $\|v-\varphi\|=\rho$.
Since $\rho $  is arbitrary, we can find critical points of $I$ in any ball centered in  $\varphi$, which contradicts Lemma \ref{kaji}.\\
Therefore,  the Mountain Pass geometry on $\mathcal{N}$  is verified, and 
knowing that the functional $I$ satifies $(PS)_{c}$,
there exists a critical point $u^*$ of the functional $I$ constrained to $\mathcal{N}$, and $ \underline d < c \leq 0$. Recalling that $\mathcal{N}$ is a natural constraint, then $u^*$ is a solution.
\end{proof}
Note that the solution just found may be the trivial one.  
In this case, knowing it is a Mountain Pass solution constrained to $\mathcal{N}$, with Morse index at least one, we would be able to conclude that zero has at least one unstable direction in the parabolic setting.

The next theorem gives a sufficient condition for such min-max solution $u^*$ not to be trivial.
\begin{theorem}
Assume the hypotheses of Theorem \ref{MPT}, and $\lambda_1(\Omega)< \lambda_2(\Omega)< \lambda <\lambda_1(\Omega_0)$, then  $\underline d <I(u^*) <0$ and $u^*$ is a sign-changing solution of problem \eqref{eq2elliptic}.
\end{theorem}
\begin{proof}
First we want to show that $I(u^*)= c < 0$, which gives $u^* \not = 0$.
In order to do so, we consider the positive (normalized in $L^2(\Omega)$) first eigenfunction of  $-\Delta$ in $\Omega$, denoted by $\phi_1$, associated with the eigenvalue $ \lambda_1(\Omega)$, a normalized eigenfunction $\phi_2$, associated with the second eigenvalue $\lambda_2(\Omega)$, $\phi_1^0$ the positive (normalized) eigenfunction associated with the first eigenvalue $\lambda_1(\Omega_0)$, and a normalized eigenfunction $\phi_2^0$, associated with the second eigenvalue $\lambda_2(\Omega_0)$. Note that the supports of $\phi_i^0$, $i=1,2$, are subsets of $\overline {\Omega_0}$. 
Moreover, $\langle \phi_1, \phi_2\rangle=0$, $\langle \phi_1^0, \phi_2^0\rangle=0$,  and $\langle \phi_i, \phi_j^0\rangle=0$, for $i,j=1,2$.\\
In order to construct a convenient path in $\Gamma$ not passing through zero we define
$w=t_1(\phi_1+ \varepsilon \phi_1^0)+t_2(\phi_2+ \varepsilon \phi_2^0)$,
with constants $t_1, t_2>0$, and  for some $\varepsilon >0$ to be chosen sufficiently small.
Using that $b(x) \leq 0$, there is a positive constant $C$ such that 
\begin{align*}
	I(w)&=  \frac{t_1^2}{2} \int_{\Omega}   | \nabla \phi_1 |^2 -\lambda  \phi_1^2 +
	\varepsilon^2 \frac{t_1^2 }{2} \int_{\Omega}   | \nabla \phi_1^0 |^2 -\lambda  (\phi_1^0)^2 \\
	&+\frac{t_2^2 }{2} \int_{\Omega}   | \nabla \phi_2 |^2 -\lambda  \phi_2^2 +
	\varepsilon^2 \frac{t_2^2}{2} \int_{\Omega}   | \nabla \phi_2^0 |^2 -\lambda  (\phi_2^0)^2 \\
	&+ \frac{1 }{2} \sum_{i=1}^{2} \sum_{j=1}^{2} t_i t_j
	\left\{
	\int_{\Omega} \nabla (\phi_i + \varepsilon \phi_i^0 ) \nabla (\varepsilon\phi_j^0) 
	-\lambda 
	(\phi_i + \varepsilon \phi_i^0 ) (\varepsilon\phi_j^0) 
	\right\}\\
	&- \frac{1}{\nu+1}\int_{\Omega} b(x) |w|^{\nu+1}\\
	& \leq 
	\frac{t_1^2}{2} (\lambda_1(\Omega)-\lambda)+
	\frac{t_2^2}{2} (\lambda_2(\Omega)-\lambda)
	+ \varepsilon^2 \frac{t_1^2}{2}(\lambda_1^0(\Omega)-\lambda)+
	\varepsilon^2 \frac{t_1^2}{2}(\lambda_2^0(\Omega)-\lambda)\\
	& + C \left\{ \varepsilon \Vert w\Vert^2  + \Vert w\Vert ^{\nu +1}\right\}.
\end{align*}
Since $\Vert w\Vert \leq \sqrt{t_1^2+t_2^2 + C\varepsilon^2}$, taking $t_1, t_2 > 0$ and $\varepsilon$ sufficiently small, recalling $\nu+1 >2$, 
and using the hypothesis $\lambda_1(\Omega)< \lambda_2(\Omega)< \lambda$,
we obtain
\[
I(w) \leq \frac{(t_1^2 +t_2^2)}{2}\max \{\lambda_1(\Omega)-\lambda, \lambda_2(\Omega)-\lambda\} + O(\varepsilon(t_1^2 +t_2^2))\leq -\delta_1 <0,
\]
for some constant $\delta_1>0$.\\
Now, let $w_1:= t_1(\phi_1+ \varepsilon \phi_1^0)$ and $w_2:=t_2(\phi_2+ \varepsilon \phi_2^0)$, and $w_\theta := \cos(\theta)w_1+ \sin(\theta)w_2$,
so that $w_{\pi/4}= \frac{\sqrt 2}{2} w$ and  $\Vert w_{\pi/4} \Vert= \frac{\sqrt 2}{2} \Vert w \Vert$ and
for some constant $\delta_2>0$ and for all $\theta \in [0, \pi]$,
\begin{align}\label{Iwtheta}
	\nonumber
	I(w_\theta)&\leq 
	\frac{t_1^2}{2} \cos^2(\theta)(\lambda_1(\Omega)-\lambda)+
	\frac{t_2^2}{2} \sin^2(\theta) (\lambda_2(\Omega)-\lambda)\\
	\nonumber
	&
	+ \varepsilon^2  \frac{t_1^2}{2}\cos^2(\theta)(\lambda_1^0(\Omega)-\lambda)+
	\varepsilon^2  \frac{t_1^2}{2}\sin^2(\theta)(\lambda_2^0(\Omega)-\lambda)\\
	& + C \left\{ \varepsilon \Vert w_\theta\Vert^2  + \Vert w_\theta\Vert ^{\nu +1}\right\}
	 \leq -\delta_2 <0.
\end{align}

Finally, define the path in $H^1_0(\Omega)$ by:
\begin{equation}\label{path}
	\gamma (s):=
	\begin{cases}
		\left[ (1-3s)\varphi + 3s (w_1)\right],\quad s\in [0, 1/3]\\
		w_{\theta(s)}, \quad s \in [1/3, 2/3]\;\;\text{and } \; \theta (s)= 3(s-1/3)\pi,\\
		\left[ 3(1-s)(-w_1) + 3(s-2/3) (-\varphi)\right], \quad s \in [2/3, 1],
	\end{cases}
\end{equation}
which can be projected on $\mathcal{N}$ by the multiplication $\tau(s) \gamma(s)$, with
$$\tau(s)= \displaystyle  \left [
\frac{\int_{\Omega}   | \nabla \gamma(s) |^2 -\lambda  (\gamma(s))^2}{\int_{\Omega}  b(x) |\gamma(s)|^{\nu+1}} 
\right].$$

Indeed, using that $\varphi \in \mathcal{S}^+$, $\varphi $ is a solution of \eqref{eq2elliptic}, and the definition of $w_1$ which involves the eigenfunction $\phi_1$, simple calculations for $s \in [0,1/3]$ yield

\begin{align*}
	{\int_{\Omega}   | \nabla \gamma(s) |^2 -\lambda  (\gamma(s))^2}&=
	(1-3s)^2 
	\int_{\Omega}   | \nabla  \varphi |^2 -\lambda  \varphi^2
	+(3s)^2
	\int_{\Omega} | \nabla  w_1 |^2 -\lambda  w_1^2\\
	&+
	2(1-3s)(3s)
	\int_{\Omega} \nabla \varphi \nabla w_1 - \lambda \varphi w_1
	<0.
\end{align*}
Also,  since $supp\{\varphi\} \cap (\Omega \setminus \overline{\Omega_0}) \not = \emptyset$,
and $supp \{w_1\} \cap (\Omega \setminus \overline{\Omega_0}) \not = \emptyset$, then
$
{\int_{\Omega}  b(x) |\gamma(s)|^{\nu+1}}= {\int_{\Omega \setminus \Omega_0}  b(x) |\gamma(s)|^{\nu+1}} <0,
$
for all $s \in [0,1/3]$.
Henceforth, $\gamma(s) \in \mathcal{S}^+$, for all $s \in [0,1/3]$, yielding
$\max_{0 \leq s \leq 1/3} I(\gamma(s) )<0$. Analogously for $s \in [2/3, 1]$.\\
The second segment of the path $\gamma$, for each $s \in [1/3, 2/3]$, by \eqref{Iwtheta} also satisfies both
\begin{equation*}
	{\int_{\Omega}   | \nabla \gamma(s) |^2 -\lambda  (\gamma(s))^2}=
	\int_{\Omega}   | \nabla  w_\theta |^2 -\lambda  w_\theta^2
	\leq -\delta_2
	<0,
\end{equation*}
\begin{equation*}
	{\int_{\Omega}  b(x) |\gamma(s)|^{\nu+1}}= {\int_{\Omega \setminus \Omega_0}  b(x) |\gamma(s)|^{\nu+1}} <0.
\end{equation*}
The second inequality uses the fact that $\lambda_2(\Omega)< \lambda$, which implies that 
$supp \{w_2\} \cap (\Omega \setminus \overline{\Omega_0}) \not = \emptyset$.
This shows that on the continuous path $\tau(s)\gamma \in \mathcal{S}^+$ it holds by Remark \ref{signI} that there is a negative upper bound $I(\tau(s) \gamma(s))\leq \max_{0 \leq t \leq 1}I(\tau(s) \gamma(s)) < 0$, for all $s\in [0,1]$.
By the definition of the min-max level $c$, it follows that $I(u^*)=c <0$. \\
%
In order to conclude, suppose by contradiction that $u^*$ is w.l.o.g. non-negative and nontrivial.
If the subset $\tilde  \Omega \subset \Omega$, on which $u^*=0$ is non-empty, then its boundary points $\partial \tilde \Omega \subset \Omega$. 
Let $x_0$ be a point in $ \partial \tilde \Omega \subset \Omega$,  $u^*(x_0)=0$.
Moreover, $u^* \in C^1(\Omega)$ (see Remark \ref{BrezisKato}), hence $\partial \tilde \Omega$ is smooth enough, and compact.
Because of the higher power of the nonlinear term in $f(x,u)$, it holds that near the points of
$\partial \tilde \Omega$ we have $-\Delta u= \lambda u + o(|u|) >0$.  Hopf Lemma 
gives that $Du^*(x_0)\not =0$, which is impossible in an interior minimum point.
Therefore, $u^* >0$, which is impossible by the uniqueness of the positive solution.  This leads to the conclusion that $u^*$ changes sign.
\end{proof}


%

\section{The parabolic problem}

The local existence in time for equation \eqref{intro} follows directly from the fact that $f(\lambda, x, u)$ is locally Lipschitz in $u$, see \cite{Hen81}. Then we have a locally defined semigroup $ u(t):= S(t, u_0)$, for $0\leq t < T_{u_0}$ and $T_{u_0}$ being the maximum time of existence.  

In addition, note that for $u_0 \in H^1_0(\Omega)$,
if we differentiate the map $t \mapsto I(u(t))$ with respect to $t$, we get
	$
\frac{d}{dt} I(u(t))= - \int_\Omega u_t^2(t) \quad \text{for all} \; t >0,
$
which implies that $I$ is decreasing along non-stationary solutions. In this case, $I$ is referred to as Lyapunov functional and the dynamical system generated by the semigroup is said to have a gradient structure. 

In order to analyze the parabolic problem, we begin by proving global existence in time $t$. 

\begin{theorem}
Let $\lambda_1<\lambda<\lambda_1(\Omega_0)$. Then the solutions of \eqref{intro} exist for all forward time. Additionally, no solution may blow-up in infinite-time (i.e. grow-up).
\end{theorem}
\begin{proof}
For any $u_0\in H_0^1(\Omega)$ we claim that the corresponding solution $S(t,u_0)$ is uniformly bounded in time. Indeed, suppose there exists $u_0$ such that $S(t,u_0)$ blows-up in finite or infinite-time. Since $\mathcal{S}^+$ and $\mathcal{N}^-$ are bounded sets in $H_0^1(\Omega)$, it should exist $\bar{t}\geq 0$ such that
$
S(t,u_0)\in \mathcal{N}^+, \forall t>\bar{t}.
$
We also get, from the gradient structure of the system that 
$
S(t,u_0)\in I^k, \forall t>\bar{t}.
$
with $k=I(S(\bar{t},u_0))$. By applying Lemma \ref{GWLemma5}, we conclude that 
$
\left\{ S(t,u_0):t\geq \bar{t} \right\}
$
is contained in a bounded subset of $H_0^1(\Omega)$, which gives a contradiction.
We conclude that any solution of \eqref{intro} remains uniformly bounded in time and, in particular, it is defined for all $t\geq 0$ (see \cite[Theorem 3.3.4]{Hen81}).

\end{proof}


%

It is known that all nonnegative  solutions of  \eqref{eq0} converge to the unique positive equilibrium $\varphi$, see \cite{Arrieta}.
In what follows we address the local evolutionary dynamics close to a stationary Mountain Pass solution, inspired by the ideas in \cite{FeMa}.

Let $\phi$ be a nontrivial stationary solution of \eqref{eq0}. Then the linearized operator at $\phi$\\
$
\mathcal{L}u=-\Delta u-f_u(\lambda, x, \phi)u
$
is self-adjoint in $L^2(\Omega)$ with domain $H_0^1(\Omega) \cap H^2(\Omega)$ and spectrum entirely composed of eigenvalues. 
We denote by $\{ \mu_i^\phi \}_{i=1}^\infty$ the nondecreasing sequence of eigenvalues of $\mathcal{L}$, repeated according to their (finite) multiplicities, and let $\{ \psi_i^\phi\}_{i=1}^\infty$ the corresponding eigenfunctions. 

We know that $\mu_i^\phi\rightarrow +\infty$ as $i\rightarrow \infty$ and, 
if we take $\phi=u^*$, then
$\mu_1^{u^*}<0$, by Theorem \ref{MPT}.
Thus we may define
\begin{equation}\label{q}
q:= \max\{ i\in\mathbb{N}: \mu_i^{u^*} \leq 0 \}.
\end{equation}
Since $\{ \psi_i^{u^*}\}_{i=1}^\infty$ form a Hilbert basis of $L^2(\Omega)$, we may decompose any $v \in  H_0^1(\Omega)$,
\begin{equation*}
v=\sum_{i=1}^q a_i\psi_i^{u^*}+\sum_{i=q+1}^\infty a_i\psi_i^{u^*}.
\end{equation*}

\begin{theorem}\label{dickstein}
Let $u^*$ be Mountain Pass solution obtained in Theorem \ref{MPT}. Then there exist initial data $u_0, v_0\in H_0^1(\Omega)$ with $u_0\in \mathcal{N}_+$ and $v_0\in \mathcal{N}_-$ with both converging to $u^*$ as time goes to infinity, i.e. $u_0,\;v_0 \in W^s (u^*)$.
\end{theorem}

In order to prove Theorem \ref{dickstein}, we need the following lemma. For simplicity we denote $\mu_i=\mu_i^{u^*}$ and $\psi_i=\psi_i^{u^*}$.

\begin{lemma}\label{ajnotequalzero}
For $u^*$ there exists $i>q$, where $q$ is given in \eqref{q}, such that
\begin{equation}\label{aj}
	a_i:=\int_{\Omega} \phi \psi_i \neq 0.
\end{equation}
 
\end{lemma}

\begin{proof}
We know that $u^*$ is a Mountain Pass critical point constrained to $\mathcal{N}$, by Theorem \ref{MPT}, hence its Morse index is at least equal to $1$ and, by definition of $q$, $\mu_i >0$ for $i > q$. 

Suppose by contradiction that $a_i=0$ for all $i>q$. Then $u^*$ may be written as
$u^*=\sum_{i=1}^{q}a_i\psi_i$. Then, since
\begin{equation*}
	\mathcal{L}(u^*)=\mathcal{L}\left( \sum_{i=1}^{q}a_i\psi_i \right)= \sum_{i=1}^{q}a_i\mathcal{L}(\psi_i)=\sum_{i=1}^{q}a_i\mu_i\psi_i,
\end{equation*}
we get 
\begin{equation*}
	\langle \mathcal{L}(u^*), u^* \rangle= \langle \sum_{i=1}^{q}a_i\mu_i\psi_i, \sum_{i=1}^{q}a_i\psi_i \rangle = \sum_{i=1}^{q}a_i^2\mu_i\leq 0.
\end{equation*}
On the other hand
$
	\langle \mathcal{L}(u^*), u^* \rangle= I^{\prime \prime}(u^*) {u^*}^2>0,
	$
since $u^*\in \mathcal{S}^+$.
\end{proof}

Now we can proceed to prove of Theorem \ref{dickstein}.

\begin{proof}
We denote by $X_1$ the finite dimensional subspace of $H_0^1(\Omega)$ spanned by $\{ \psi_i: 1\leq i\leq q \}$ and by $X_2$ the infinite dimensional subspace of $H_0^1(\Omega)$ spanned by $\{ \psi_i: i > q \}$. It is known that the local stable manifold of $u^*$, denoted by $W^s_{loc}(u^*)$, is tangent to $X_2$ at $u^*$. In addition, there exists a neighborhood $V$ of $0$ in $X_2$ and a $C^1$ map $h:V\rightarrow X_1$ such that
\begin{equation}\label{stableman}
	W^s_{loc}(u^*)=\{ u^*+\eta+h(\eta): \eta \in V \}.
\end{equation}

By Lemma \ref{ajnotequalzero}, there exists at least one $i>q$ such that $a_i\neq 0$. Notice that, for any $\psi \in H_0^1$, it holds that
\begin{eqnarray*}
	\langle  \mathcal{L} u^*,\psi \rangle_{L_2(\Omega)} =-\int_{\Omega} \Delta u^* \psi- f_u(\lambda, x, u^*) u^*\psi
	=\int_{\Omega} \nabla u^* \nabla\psi -f_u(\lambda, x, u^*) u^*\psi
\end{eqnarray*}
and
\begin{align*}
	\langle J'(u^*), \psi \rangle_{L_2(\Omega)}&= 2 \int_{\Omega} \nabla u^* \nabla \psi - \int_{\Omega} f_u(\lambda, x, u^*) u^* \psi - \int_{\Omega} f(\lambda, x, u^*) \psi\\
	&=\int_{\Omega} \nabla u^* \nabla \psi -\int_{\Omega} f_u(\lambda, x, u^*) u^*\psi.
\end{align*}
Therefore, for $\psi=\psi_i$, the normalized eigenfunction with $i>q$, $\langle J'(u^*), \psi_i \rangle_{L_2(\Omega)}= \langle  \mathcal{L} u^*,\psi \rangle_{L_2(\Omega)}=a_i \mu_i\neq 0.$

It follows from \eqref{stableman} that
$
u_0:= u^*+\epsilon \psi_i+h(\epsilon \psi_i) \in W^s_{loc}(u^*)
$
if $|\epsilon|\neq 0$ is small enough. Therefore, the corresponding solution $u(x,t)$ converges to $u^*$ as $t\rightarrow \infty$. Now we claim that $J(u_0)$ has the same sign as $\epsilon$, for small $|\epsilon|$. Indeed, by Taylor's expansion at $u^*$, since $J(u^*)=0$ then
$
J(u_0)=J(u^*)+J'(u^*)(\epsilon \psi_i)+R_\epsilon=\epsilon \mu_i a_i+R_\epsilon.
$
The statement follows since we can assume, without loss of generality that $a_i>0$.
\end{proof}

\medskip
\medskip


\noindent Instituto de Ci\^encias Exatas, Departamento de Matem\'atica,
Universidade de Bras\'ilia,
70910-900 Brasília - DF, Brazil,
\noindent\texttt{lilimaia@unb.br}
\bigskip

\noindent Instituto de Matem\'atica,
Universidade Federal do Rio de Janeiro,
21941-909 Rio de Janeiro - RJ, Brazil,
\noindent\texttt{jfernandes@im.ufrj.br}

\end{document}